\documentclass[12pt,a4paper, english]{article}

\usepackage[T1]{fontenc}
\usepackage{babel}		
\usepackage[utf8]{inputenc}
\usepackage{amsmath}  
\usepackage{amsthm}		
\usepackage{amsfonts}		
\usepackage{amssymb} 
\usepackage{enumerate} 
\newtheorem{teorema}{Theorem}
\newtheorem{lem}[teorema]{Lemma}
\newtheorem{prop}[teorema]{Proposition}
\newtheorem{cor}[teorema]{Corollary}
\newtheorem{athm}{Theorem}

\theoremstyle{definition}

\newtheorem{nota}[teorema]{Remark}
\newtheorem{ex}[teorema]{Example}

\newtheorem{question}[teorema]{Question}

\newcommand{\bcdot}{\boldsymbol{\cdot}}

\DeclareMathOperator{\Aut}{Aut}

\DeclareMathOperator{\Ker}{Ker}

\DeclareMathOperator{\Z}{Z}

\title{On left nilpotent skew braces of class~$2$}
\author{A. Ballester-Bolinches%
\thanks{Departament de Matem\`atiques, Universitat de Val\`encia, Dr.\ Moliner, 50, 46100 Burjassot, Val\`encia, Spain; \texttt{Adolfo.Ballester@uv.es}, \texttt{Vicent.Perez-Calabuig@uv.es}; ORCID 0000-0002-2051-9075, 0000-0003-4101-8656.}
\and L. A. Kurdachenko\thanks{Department of Algebra and Geometry, Oles Honchar Dnipro National University, Dnipro 49010, Ukraine; \texttt{lkurdachenko@gmail.com}; ORCID 0000-0002-6368-7319.}\ \thanks{Part of the research of this author has been carried out in the Departament de Matem{\`a}tiques, Universitat de Val{\`e}ncia; Dr.\ Moliner, 50; 46100 Burjassot, Val\`encia, Spain.} \and V. P\'erez-Calabuig\addtocounter{footnote}{-3}\footnotemark}

\begin{document}
\date{}
\maketitle

\begin{abstract}
The main objective of this article is to initiate a detailed structure theory of left nilpotent skew braces $B$ of class~$2$, i.e. skew braces with $B^3 = 0$. We prove that if $B$ is of nilpotent type, then $B$ is centrally nilpotent. In fact, we show that $B$ is right nilpotent of class at most $2+mr$, i.e. $B^{(2+mr+1)} = 0$, where $m$ and $r$ are the nilpotency classes of the additive group of $B$ and $B^2$, respectively. If $B$ is of abelian type, then $B$ is actually right nilpotent of class~$3$, i.e. $B^{(4)} = 0$, and this bound is best possible. 
  

\end{abstract}

\section{Introduction}
Skew braces were introduced in the context of analysing the class of bijective non-degenerate combinatorial solutions (solutions, for short) of the Yang-Baxter equation, a fundamental equation in mathematical physics with broad implications across various areas. It has become increasingly evident that a deeper understanding of algebraic properties of skew braces is crucial for reveal insights into the classification problem of solutions. Among the various structural properties explored, nilpotency of skew braces has garnered significant attention, particularly due to its connection with a well-studied class of solutions known as multipermutation solutions, or solutions that can be recursively retracted until getting a trivial solution after finitely many steps.

A (left) \emph{skew brace} is a set $B$ with two group structures, $(B,+)$ and $(B,\cdot)$, satisfying the (left) distributivity property: $a \cdot (b+c) = a \cdot b - a + a \cdot c$ for every $a,b,c\in B$. 
If $\mathfrak{X}$ is a class of groups and $(B,+)\in \mathfrak{X}$, $B$ is said to be of \emph{$\mathfrak{X}$-type} (Rump's \emph{braces} introduced in his seminal paper~\cite{Rump07} are skew braces of abelian type).
The nilpotent type universe is particularly noteworthy when studying nilpotency of skew braces: a solution is multipermutation if, and only if, its associated skew brace structure is right nilpotent of nilpotent type (see 
 \cite{CedoJespersKubatVanAntwerpenVerwimp23,CedoSmoktunowiczVendramin19}); and a finite skew brace $B$ of nilpotent type is left nilpotent if, and only if, $(B,\cdot)$ is nilpotent (see \cite{CedoSmoktunowiczVendramin19, Smoktunowicz18-tams}). It is natural, therefore, to ask under which conditions left nilpotency implies right nilpotency, or vice versa.  

Smoktunowicz in~\cite{Smoktunowicz18-tams} made significant strides by 
introducing \emph{strongly nilpotent} skew braces---those that are both left and right nilpotent. In the nilpotent type case, it turns out that strong nilpotency coincides with central nilpotency, a property that can be characterised via the centre of a skew brace (see~\cite{BonattoJedlicka23,JespersVanAntwerpenVendramin23}). Central nilpotency plays a crucial role in the structure theory of skew braces (see~\cite{BallesterEstebanFerraraPerezCTrombetti25-central} for a detailed study), enabling a more tractable analysis for the description of finitely-generated skew braces, a challenging problem which is a long way off from being solved. Some partial results are obtained in the abelian type case: there are  complete descriptions of centrally nilpotent one-generated skew braces with low nilpotency classes (see~\cite{BallesterEstebanKurdachenkoPerezC25-onegenerated, DixonKurdachenkoSubbotin25-preprint, KurdachenkoSubbotin24}, respectively, for right nilpotency class~$2$, left nilpotency class~$2$, and central nilpotency class~$3$).


As it is claimed in~\cite{Smoktunowicz18-tams}, not every right nilpotent skew brace of class~$2$ is centrally nilpotent and not every left nilpotent skew brace of class~$3$ is centrally nilpotent (see \cite[Examples 2 and 3]{Rump07} in the abelian type case). From these examples the following question remained open: let $B$ be a skew brace of nilpotent type which is left nilpotent of class~$2$. Is $B$ centrally nilpotent? Our main result answers this question.

\begin{athm}
\label{teo:A}
Let $B$ be a skew brace of nilpotent type such that $B$ is left nilpotent of class~$2$. Then, $B$ is right nilpotent of class at most $2+mr$, i.e. $B^{(2+mr+1)} = 0$, where $m$ and $r$ are the nilpotency classes of the additive group of $B$ and $B^2$, respectively. In particular, $B$ is centrally nilpotent. 
\end{athm}

\begin{cor}
Every solution of the Yang-Baxter equation with associated skew brace structure of nilpotent type, and left nilpotent of class~$2$, is a multipermutation solution.
\end{cor}

\begin{cor}
\label{cor:abeliancase}
Let $B$ be a skew brace of abelian type such that $B$ is left nilpotent of class~$2$. Then, $B$ is right nilpotent of class~$3$, i.e. $B^{(4)} = 0$.
\end{cor}

\begin{nota}
The nilpotent type hypothesis in Theorem~\ref{teo:A} is necessary as Example~\ref{ex:nonilpotent} below shows. In a forthcoming work we wil describe the structure of finitely generated left nilpotent skew braces of left nilpotent class~$2$.
\end{nota}

\section{Preliminaries}
Let $B$ be a skew brace. The distributivity property between the two group structures in $B$ yields a common identity element $0 \in B$. It also leads to an action given by a homomorphism $\lambda\colon (B,\cdot) \rightarrow \Aut(B,+)$, $a\mapsto \lambda_a$, with $\lambda_a(b) = -a + ab$ for every $a,b\in B$. We write products in $(B,\cdot)$ by juxtaposition and preceding sums. We shall use $[\, ,\,]_+$ and $[\, ,\,]_{\bcdot}$ to respectively denote additive and multiplicative commutators in $(B,+)$ and $(B,\cdot)$. Given a subset $X$ of $B$, $\langle X \rangle_+$ denotes the additively generated subgroup in $(B,+)$. 

Nilpotency concepts of skew braces arose from the so-called \emph{star product} in $B$: $a\ast b = -a + ab - b = \lambda_a(b) - b$ for every $a,b\in B$ (star products act before products and sums). If $X, Y \subseteq B$, then $X \ast Y = \langle x \ast y\mid x \in X,\, y \in Y \rangle_+$. An \emph{ideal} $I$ of $B$ is a normal subgroup $(I,+) \unlhd (B,+)$ such that $I \ast B, B \ast I \subseteq I$; or equivalently, a $\lambda$-invariant normal subgroup $(I,+)\unlhd (B,+)$ such that $(I,\cdot)\unlhd (B,\cdot)$. This yields $b+I = bI$ for every $b \in B$, so that $(B/I,+,\cdot)$ has a quotient skew brace structure.

Following \cite{Rump07}, we can define a left (resp. right) iterated series:
\begin{align*}
(L)\  &B^1 = B \geq B^2 = B \ast B \geq \ldots \geq B^{n+1} = B \ast B^n \geq \ldots\\
(R)\  &B^1 = B \geq B^{(2)} = B \ast B \geq \ldots \geq B^{(n+1)} = B^{(n)} \ast B \geq \ldots
\end{align*}
A skew brace $B$ is said to be \emph{left (resp. right) nilpotent of class $n$} if $n$ is the smallest natural such that $B^{n+1} = 0$ (resp. $B^{(n+1)} = 0$). It is well-known that each term of the right series is an ideal of $B$ ; in particular, $B^2 = B \ast B$ is an ideal of $B$. 

In~\cite{BonattoJedlicka23}, the \emph{centre} of a skew brace $B$ is defined as
\[ \zeta(B) = \{b \in B \mid a \ast b = b \ast a = [a, b]_+ = 0, \ \text{for all $a\in B$}\}\]
Then, $B$ is defined to be \emph{centrally nilpotent} if there exists a chain of ideals
\[0= I_0 \leq I_1 \leq \ldots \leq I_n = B\]
such that $I_j/I_{j-1} \leq \zeta(B/I_{j-1})$ for every $1\leq j \leq n$. Corollary~2.15 in~\cite{JespersVanAntwerpenVendramin23} states that if $B$ is of nilpotent type, then $B$ is centrally nilpotent if, and only if, $B$ is left and right nilpotent.


\begin{nota}
The following property can easily be checked:
\[ (ab) \ast c = a \ast (b \ast c) + b \ast c + a \ast c, \quad \text{for every $a,b,c\in B$.}\]
\end{nota}

The following lemma follows by applying $B^3 = 0$ to the previous property.

\begin{lem}
\label{lema:propietats}
Let $B$ be a skew brace such that $B^3 = 0$. Then:
\begin{enumerate}
\item For every $c \in B^2$ and every $a \in B$, $ac = a+c$ . In particular, $c^{-1} = -c$, for every $c\in B^2$.
\item For every $a,b,x \in B$, $(ab) \ast x = b \ast x + a \ast x$. In particular, $a^{-1}\ast x = -a \ast x$.
\item For every $a,b\in B$, $[a,b]_{\bcdot} \ast x  = [-b \ast x, -a\ast x]_+$.
\item If $c\in B^2$, then $(a+c) \ast x = c \ast x + a \ast x$ for every $a,x \in B$.
\end{enumerate}
\end{lem}

\section{Proof and consequences of Theorem~\ref{teo:A}}

\begin{proof}[Proof of Theorem~\ref{teo:A}]
Since $B^3 = 0$, we see that $B^2 = B \ast B$ is a trivial skew brace, that is $(B^2,+) = (B^2,\cdot)$ in the sense that $cd = c+d$ for every $c,d\in B^2$. Moreover, every subgroup $(X,+)$ of $(B^2,+)$ is $\lambda$-invariant as $b \ast x = 0$ for every $b\in B$ and every $x\in X$.

Assume that $(B,+)$ is nilpotent of class $m\in \mathbb{N}$, i.e. the additive lower central series $\{\gamma_{n,+}(B)\}_{n\in \mathbb{N}}$ of $B$ reaches firstly $0$ at $\gamma_{m+1,+}(B) = 0$. Moreover, assume that $r\in \mathbb{N}$ is the nilpotent class of $(B^2,+)$, i.e. the additive upper central series of $B^2$, $\{\Z_n(B^2,+)\}_{n\in \mathbb{N}}$, reaches firstly $B^2$ at $\Z_{r}(B^2,+) = B^2$. 

Let $n\in \mathbb{N}$. Since $\Z_n(B^2,+)$ is a characteristic subgroup of $(B^2,+) = (B^2,\cdot)$, it holds that $\Z_n(B^2,+) \unlhd (B,+)$ and $\Z_n(B^2,\cdot)\unlhd (B,\cdot)$, as $B^2$ is an ideal of~$B$. Moreover, $\Z_n(B^2,+) \leq (B^2,+)$ is $\lambda$-invariant, and therefore, $\Z_n(B^2,+)$ is an ideal. Thus, we can consider the $\lambda$-action restricted to the quotient $B/\Z_n(B^2,+)$:
\[ \lambda^{(n)}\colon (B,\cdot) \longrightarrow \Aut\Big(B/\Z_n(B^2,+),+\Big),\]
with $\lambda^{(n)}_{a}\big(b+\Z_n(B^2,+)\big) = \lambda_a(b) + \Z_n(B^2,+)$, for every $a,b\in B$.

We claim that for every $n\in \mathbb{N}$, $S_n:= \Ker \lambda^{(n)}\cap B^2$ is an ideal of $B$. We can see that
\begin{align*}
\Ker\lambda^{(n)} & = \{a \in B\mid \lambda_a(b) + \Z_n(B^2,+) = b+\Z_n(B^2,+), \, \text{for all $b\in B$}\} \\
& = \{a \in B\mid a \ast b + \Z_n(B^2,+) = \Z_n(B^2,+), \, \text{for all $b\in B$}\}.
\end{align*}
Since $S_n \subseteq B^2$, $(S_n,\cdot) = (S_n,+)$ is $\lambda$-invariant, and $(S_n,\cdot) \unlhd (B,\cdot)$ as both $(\Ker \lambda^{(n)},\cdot)$ and $(B^2,\cdot)$ are normal subgroups of $(B,\cdot)$. Let $s \in S_n$ and $b\in B$. Certainly, $b+s-b \in B^2$, and 
\[ bs + \Z_n(B^2,+) = (bsb^{-1})b + \Z_n(B^2,+) = bsb^{-1} + b + \Z_n(B^2,+),\]
since $bsb^{-1} \in \Ker \lambda^{(n)}$. Thus, for every $a\in B$, Lemma~\ref{lema:propietats} yields
\begin{align*}
(b+s-b)\ast a + \Z_n(B^2,+) & = (bs - b) \ast a + \Z_n(B^2,+) = \\
& = (bsb^{-1})\ast a + \Z_n(B^2,+) = \\
& = -b\ast a + s \ast a + b \ast a + \Z_n(B^2,\!+) =  \Z_n(B^2,\!+)
\end{align*}
because $s \ast a + \Z_n(B^2,+) = \Z_n(B^2,+)$. Hence, $b+s-b\in S_n$ and the claim holds.

Now, let $c \in B^2$ and $b\in B$. Applying Lemma~\ref{lema:propietats}, it follows that
\begin{align}
c \ast b & = -c + cb - b = -c + bb^{-1}cb -b = -c + b + b^{-1}cb - b =  \nonumber \\
& = -c + b + cc^{-1}b^{-1}cb - b =  -c + b + c + [c^{-1},b^{-1}]_{\bcdot} - b = \nonumber \\
& = [-c,b]_+ + \big[b, [c^{-1},b^{-1}]_{\bcdot}\big]_+ + [c^{-1},b^{-1}]_{\bcdot}  \label{eq:c*b}
\end{align}
Observe that $[c^{-1},b^{-1}]_{\bcdot} \in S_{r-1}$, as $[c^{-1},b^{-1}]_{\bcdot} \in B^2$, and for every $a \in B$, by Lemma~\ref{lema:propietats}, it holds
\[ [c^{-1},b^{-1}]_{\bcdot} \ast a + S_{r-1}= [b \ast a, c \ast a]_+ + S_{r-1} = S_{r-1}\]
because  $c\ast a, b \ast a \in B^2$ and $\Z_r(B^2,+)/\Z_{r-1}(B^2,+) = B^2/\Z_{r-1}(B^2,+)$. Applying~\eqref{eq:c*b}, we get that $c \ast b + S_{r-1} = [-c,b]_+ + S_{r-1}$. Thus, for every $b_1, \ldots, b_m\in B$, we have that
$(\cdots ((c \ast b_1) \ast b_2) \ast \cdots \ast b_{m-1}) \ast b_m + S_{r-1} =$
\[ \Big[-\big[\ldots \big[-[-c,b_1]_+, b_2\big]_+, \ldots, b_{m-1}\big]_+, b_m\Big]_+ + S_{r-1} \in \gamma_{m+1}(B) + S_{r-1} = S_{r-1}.\]
Hence, it follows that $B^{(2+m)}\subseteq S_{r-1}$.

We claim that $B^{(2+mk)} \subseteq S_{r-k}$ for each $1\leq k \leq r$. Assume that it holds for some $1 \leq k-1 < r$, and let $d \in B^{(2+m(k-1))} \subseteq  S_{r-k+1}$ and $b\in B$. Again, by~\eqref{eq:c*b}, it holds that 
\[ d \ast b = [-d,b]_+ + \big[b, [d^{-1},b^{-1}]_{\bcdot}\big]_+ + [d^{-1},b^{-1}]_{\bcdot}.\]
Since $d \in S_{r-k+1}$, it follows that $d \ast a \in \Z_{r-k+1}(B^2,+)$ for every $a \in B$. Thus, $[d^{-1},b^{-1}]_{\bcdot} \in S_{r-k}$ because, by Lemma~\ref{lema:propietats}, 
\[ [d^{-1},b^{-1}]_{\bcdot} \ast a + \Z_{r-k}= [b\ast a, d \ast a]_+ + \Z_{r-k}(B^2,+) = \Z_{r-k}(B^2,+)\]
Thus, $d \ast b + S_{r-k} = [-d,b]_+ + S_{r-k}$, and therefore, 
\[ (\cdots ((d \ast b_1) \ast b_2) \ast \cdots \ast b_{m-1}) \ast b_m + S_{r-k} = S_{r-k},\quad \text{for every $b_1, \ldots, b_m\in B$.}\]
Hence, $B^{(2+mk)} \subseteq S_{r-k}$ and the claim holds. 

In particular, $B^{(2+ m\cdot r)} \subseteq S_0 = \Ker \lambda^{(0)} \cap B^2 = \Ker \lambda \cap B^2$. Hence, $B^{(2+m\cdot r + 1)} = 0$, and $B$ is right nilpotent. Since $B$ is of nilpotent type, 
we conclude that $B$ is centrally nilpotent.
\end{proof}

Let $G$ and $H$ be groups such that $G$ acts on $H$ via a homomorphism $\varphi\colon g\in G \mapsto \lambda_g\in \Aut(H)$. A bijective map $\delta \colon G \rightarrow H$ is a \emph{derivation} associated with $\lambda$, if $\delta(xy) = \delta(x)\lambda_x(\delta(y))$ for every $x,y \in G$. It is well known that the existence of a bijective derivation $\delta\colon (G,\cdot) \rightarrow (B,+)$ associated with a homomorphism $\lambda \colon g \in (G, \cdot) \mapsto \lambda_g\in \Aut(B,+)$ provides a skew brace structure $(B,+,\cdot)$ with the product given by $a\cdot b = \delta(\delta^{-1}(a)\delta^{-1}(b)) = a + \lambda_{\delta^{-1}(a)}(b)$ for every $a,b\in B$.

The following example shows that the nilpotency type hypothesis of Theorem~\ref{teo:A} is necessary.
\begin{ex}
\label{ex:nonilpotent}
Let $(B,+) = \langle \sigma, \tau \mid 3\sigma = 2\tau = 0, \sigma + \tau = \tau + 2\sigma\rangle$ be an additive group isomorphic to the symmetric group of degree~$3$. Let $(G,\cdot) = \langle g \rangle$ be a cyclic group of order $6$. We can define a homomorphism $\varphi \colon G \rightarrow \Aut(B,+)$ given by $\varphi_g(\sigma) = \sigma$ and $\varphi_g(\tau) = \sigma + \tau$. It turns out that $\Ker \lambda = \langle g^3\rangle$, so that $\lambda_{g}^{-1}(\sigma) = \sigma$ and $\lambda_g^{-1}(\tau) = 2\sigma + \tau$. Consider the bijection $\delta \colon (G,\cdot) \rightarrow (B,+)$ given by
\[ 1 \mapsto 0, \quad g \mapsto 2\sigma + \tau, \quad g^2 \mapsto 2\sigma, \quad g^3 \mapsto \tau, \quad g^4\mapsto \sigma, \quad g^5\mapsto \sigma + \tau.\]
It is easy to check that $\delta$ is a derivation associated with $\varphi$, and therefore, it provides a skew brace structure $(B,+,\cdot)$ with non-nilpotent additive group.

Observe that $\lambda_{\sigma} = \lambda_{2\sigma + \tau} = \varphi_g$, $\lambda_{2\sigma} = \lambda_{\sigma + \tau} = \varphi_g^{-1}$, and $\lambda_{\tau} = \lambda_0 = \operatorname{id}_B$. Thus, $\lambda_x(\langle \sigma\rangle + \tau) = \langle \sigma \rangle + \tau$ for every $x \in B$. Recall that $a \ast b = \lambda_a(b) - b$. Hence, it is a routine to check that $B^2 = B^{(2)} = B \ast B = \langle \sigma \rangle_+$, and therefore, $B^3 = B \ast B^2 =  0$ and $B^{(3)} = B^{(2)} \ast B = B^{(2)}$.
\end{ex}

If $B$ is a skew brace of abelian type, then $m = r = 1$ in Theorem~\ref{teo:A}. Therefore, if $B^2= 0$, it follows that $B^{(4)} = 0$. This yields Corollary~\ref{cor:abeliancase}.


The following example shows that $3$ is the best possible upper bound for the right nilpotency class in the abelian type case.

\begin{ex}
Let $(B,+) = \langle a, b \mid 4a = 2b = 0, a+b = b+a\rangle$ be an additive group isomorphic to $C_4\times C_2$. Let $(G,\cdot) = \langle \sigma, \tau \mid \sigma^4= \tau^2= 1, \sigma\tau = \tau\sigma^3 \rangle$ be a dihedral group of order $8$. We can define a homomorphism $\varphi \colon G \rightarrow \Aut(B,+)$ given by 
\[ \begin{array}{ll}
\varphi_\sigma(a) = 3a + b,& \varphi_\tau(a) = a+b,\\
\varphi_\sigma(b) = b, & \varphi_\tau(b) = b,
\end{array}\]
so that $\Ker \varphi = \langle \sigma^2 \rangle$. Consider the bijection $\delta \colon (G,\cdot) \rightarrow (B,+)$ given by
\[ \begin{array}{llll}
1 \mapsto 0, & \sigma \mapsto a+b, & \sigma^2 \mapsto b, & \sigma^3 \mapsto a, \\
\tau \mapsto  2a, & \sigma \tau \mapsto 3a + b, & \sigma^2\tau \mapsto 2a+b, & \sigma^3\tau \mapsto 3a.
\end{array} \]
It is easy to check that $\delta$ is a derivation associated with $\varphi$, and therefore, it provides a skew brace structure $(B,+,\cdot)$ of abelian type.

Observe that $\Ker \lambda = \langle b \rangle$, $\lambda_{a} = \lambda_{a+b} = \varphi_g$, $\lambda_{2a} = \lambda_{2a+b} = \varphi_\tau$, and $\lambda_{3a} = \lambda_{3a+b} = \varphi_\sigma \varphi_\tau = \varphi_\tau \varphi_\sigma$. Thus, the subgroup $\langle 2a, b\rangle_+$ is fixed by the $\lambda$-action. As a consequence, we have that $B^2 = B^{(2)} = B \ast B = \langle 2a,b \rangle_+$, and therefore, $B^3 = B \ast B^2 =  0$. Hence,  we can see that $B^{(3)} = B^{(2)} \ast B = \langle b \rangle_+$, and therefore, $B^{(4)} = B^{(3)}\ast B = 0$.
\end{ex}

The following question naturally arises from Theorem~\ref{teo:A}. 

\begin{question} Does exist a skew brace $B$ of nilpotent type with $B^3 = 0$, $m$ and $r$ the nilpotency classes of $(B,+)$ and $(B^2,+)$, respectively, such that $2+ mr$ is the right nilpotency class of $B$?
\end{question}

The following results show particular cases in which the upper bound for the right nilpotency class can be lower.

\begin{prop}
Let $B$ be a skew brace of nilpotent type such that $B^3 = 0$. If $B^{(k)} \subseteq \Z(B,\cdot)$ for some $k \geq 2$, then $B^{(k+m+1)} = 0$, where $m$ is the nilpotency class of $(B,+)$. In particular, if $B^{(2)}\subseteq \Z(B,\cdot)$, $B^{(2+m+1)} = 0$.
\end{prop}

\begin{proof}
Assume that $m$ is the nilpotency class of $(B,+)$. Let $c\in B^{(k)}$ and let $b\in B$. We can see that
\[ c \ast b = -c + cb - b = -c + bc -b = -c + b + c - b = [-c,b]_+\]
Therefore, for every $b_1, \ldots, b_m \in B$, $( \cdots((c\ast b_1)\ast b_2) \cdots ) \ast b_m = $
\[\Big[-\big[\ldots \big[-[-c,b_1]_+, b_2\big]_+, \ldots, b_{m-1}\big]_+, b_m\Big]_+ \in \gamma_{m+1}(B) =  0\]
i.e. $B^{(k+m+1)} = 0$.
\end{proof}

\begin{cor}
Let $B$ be a skew brace of nilpotent type such that $B^3 = 0$ and $(B,\cdot)$ is abelian. Then, $B^{(2+m+1)} = 0$, where $m$ is the nilpotency class of $(B,+)$.
\end{cor}

\section*{Acknowledgements}
The first and third authors are supported by the grant CIAICO/2023/007 from the Conselleria d'Educació, Universitats i Ocupació, Generalitat Valenciana. The second author is very grateful to the Conselleria d'Innovaci\'o, Universitats, Ci\`encia i
Societat Digital of the Generalitat (Valencian Community, Spain) and the Universitat de
Val\`encia for their financial support and grant to host researchers affected by the war in
Ukraine in research centres of the Valencian Community. The second author would also like to thank the Isaac Newton Institute for Mathematical Sciences, Cambridge, for support and hospitality during the Solidarity Supplementary Grant Program. This work was supported by EPSRC grant no EP/R014604/1. 

\bibliographystyle{plain}
\bibliography{bibgroup}

\end{document}